\documentclass[a4paper,12pt]{article}

\usepackage{amssymb}

\usepackage{amsthm}
\usepackage{amsmath}
\theoremstyle{plain}
\newtheorem{theorem}{Theorem}[section]

\newtheorem{corollary}[theorem]{Corollary}
\newtheorem{proposition}[theorem]{Proposition}

 \title{On the Lindel\"{o}f Hypothesis for the Riemann Zeta function and Piltz divisor problem }
\author{Lahoucine Elaissaoui\footnote{Department of mathematics, Faculty of sciences, Mohammed V University in Rabat, 4 Street Ibn Battouta B.P. 1014 RP, Rabat, e-mail: \texttt{l.elaissaoui@um5r.ac.ma}; \texttt{lahoumaths@gmail.com}; ORCID: 0000--0003--4600--4202}}

\begin{document}

\maketitle


\vspace*{1cm}

\begin{abstract}
In order to well understand the behaviour of the Riemann zeta function inside the critical strip, we show; among other things, the Fourier expansion of the $\zeta^k(s)$ ($k \in \mathbb{N}$) in the half-plane $\Re s > 1/2$ and we deduce a necessary and sufficient condition for the truth of the Lindel\"{o}f Hypothesis. Moreover, if $\Delta_k$denotes the error term in the Piltz divisor problem then for almost all $x\geq 1$ and any given $k \in \mathbb{N}$ we have
$$\Delta_k(x) = \lim_{\rho \to 1^-}\sum_{n=0}^{+\infty}(-1)^n\ell_{n,k}L_n\left(\log(x)\right)\rho^n  $$
where $(\ell_{n,k})_{n}$ and $L_n$ denote, respectively, the Fourier coefficients of $\zeta^k(s)$ and Laguerre polynomials. 
\end{abstract}

\textbf{ MSC 2020.}{Primary  11M06, 30B10, 30B40, 30B50, 30B30, 11N56; Secondary 11B65}

\textbf{Keywords.}{ Riemann zeta function, Lindel\"{o}f Hypothesis, Divisor problems, Stieltjes constants, Fourier series}

\section{Fourier expansion of powers of the Riemann zeta function}
\subsection{Introduction and statements}
\label{}
 The Lindel\"{o}f Hypothesis is a significant open problem in analytic number theory that concerns the growth of the Riemann zeta function $\zeta(s)$ on the critical line, $\Re s = 1/2.$ We recall that $\zeta(s)$ is initially defined for any complex number $s=\sigma + it$ in the half-plane $\sigma > 1$ by the Dirichlet series $\zeta(s) = \sum_{n\geq 1} 1/n^s$ and extends analytically, by its integral representation
 \begin{equation}
 \zeta(s) = \frac{s}{s-1} - s\int_1^{+\infty} \frac{\{x\}}{x^{s+1}} \mathrm{d}x,  
 \label{intrep}
 \end{equation}
 where $\{\cdot\}$ denotes the fractional part function, and the functional equation \cite[p. 16]{Titch}
 \begin{equation} 
\zeta(s) = \chi(s)\zeta(1-s) \quad \mbox{where} \quad \chi(s) = \pi^{s-\frac{1}{2}}\frac{\Gamma\left(\frac{1-s}{2} \right)}{\Gamma\left( \frac{s}{2} \right)}
\label{funeq}
 \end{equation}  
($\Gamma$ is the well-known Euler gamma function), to the whole complex plane except for a simple pole at $s=1.$ Thus, it is clear that $\zeta(s)$ is bounded in any half-plane $\sigma \geq \sigma_0> 1;$ and by the functional equation \eqref{funeq}, since for any bounded $\sigma$  we have \cite[p. 78]{Titch} $$|\chi(s)| \sim  \left(\frac{|t|}{2\pi}\right)^{\frac{1}{2}-\sigma}\quad \mbox{as} \quad |t| \to \infty,$$
then for all $ \sigma \leq 1-\sigma_0<0,$
$$\zeta(s) = O\left( |t|^{\frac{1}{2}-\sigma}\right) . $$
However, the order of $\zeta(s)$ inside the critical strip $0 < \sigma < 1$ is not completely understood. The Phragm\'{e}n-Lindel\"{o}f principle \cite[§9.41]{Titch2} implies that if $\zeta\left( \frac{1}{2}+it \right)=O\left(|t|^{\kappa+\varepsilon}\right),$ for any $\varepsilon >0,$ then we have
$$\zeta(s) = O\left(|t|^{2 (1-\sigma)\kappa+\varepsilon} \right), \quad \forall \varepsilon > 0, $$
 uniformly in the strip $1/2\leq \sigma < 1;$ and the order of the Riemann zeta function in the strip $0<\sigma \leq 1/2$ follows from the functional equation \eqref{funeq}. Notice that, the optimal value of $\kappa$ is not known and the best value obtained to date is due to Bourgain \cite{Bour}, that is $\kappa = 13/84;$ however, the yet unproved Lindel\"{o}f Hypothesis states that $\kappa=0.$ Actually, there are several equivalent statements to the Lindel\"{o}f Hypothesis, see for example \cite[p. 320]{Titch} and \cite{HarLit}; in particular, by combining Theorems 12.5 and 13.4 in \cite{Titch}, the Lindel\"{o}f Hypothesis holds true if and only if the integral
 \begin{equation}
 \frac{1}{2\pi} \int_{\Re s = \frac{1}{2}}\frac{|\zeta(s)|^{2k}}{|s|^2} |\mathrm{d}s| 
 \label{Nrmz}
 \end{equation}
converges for any $k \in \mathbb{N}.$ 

Recently, the author and Guennoun showed in \cite{Ela} that the values of the Riemann zeta function in the half-plane $\sigma \geq 1/2$ are encoded in the binomial transform of the Stieltjes constants $(\gamma_j)_{j\geq 0}$ (see for example \cite{Ber}); namely, for all $\sigma \geq 1/2,$ $s\neq 1,$ we have 
 \begin{equation}
 \zeta(s) = \frac{s}{s-1}+ \sum_{n=0}^{+\infty} (-1)^n\ell_n \left( \frac{s-1}{s} \right)^n 
 \label{fourexp}
 \end{equation}
 where $\ell_0=\gamma_0-1$ and $$\ell_n = \sum_{j=1}^n \binom{n-1}{j-1}\frac{(-1)^{n-j}}{j!}\gamma_j \qquad n \in \mathbb{N} $$
 is a square-summable sequence. Hence, one can deduce the estimation of the Riemann zeta function in the half-plane $\sigma > 1/2$ by studying the growth of the Fourier coefficients $(\ell_n)_{n \in \mathbb{N}_0};$ in particular, if $\ell_n = O(n^{-1+\varepsilon})$ for all $\varepsilon > 0$ as $n\to+\infty$ then the Lindel\"{o}f Hypothesis holds true. Notice that an other proof of \eqref{fourexp}, for $\sigma > 1/2,$ has been given by the author in \cite{Ela2} by proving that  $((-1)^{n-1}\ell_n)_{n \geq 0}$ are the Fourier-Laguerre coefficients of the fractional part function, $\{\cdot\},$ in the Hilbert space $$\mathcal{H}_0:= \left\{ f:(1,+\infty) \rightarrow \mathbb{C}, \quad \int_1^{+\infty}|f(x)|^2\mathrm{d}w(x) < +\infty\right\}, \quad \left( \mathrm{d}w(x) = \frac{\mathrm{d}x}{x^2}\right)$$    
associated with the orthonormal basis $(\mathcal{L}_j)_{j \in \mathbb{N}_0},$ where for each $j \in \mathbb{N}_0,$ $\mathcal{L}_j(x) = L_j(\log(x))$ and $(L_j)$ are the classical Laguerre polynomials \cite{Sz}; with respect to the inner product
$$\langle f,g \rangle = \int_1^{+\infty} f(x)\overline{g(x)}\ \mathrm{d}w(x), \qquad  f,g \in \mathcal{H}_0. $$
More generally, let for all $|s-1|\leq 1$ and any given $k \in \mathbb{N}$
\begin{equation}
(s-1)^k\zeta^k(s) = \sum_{j=0}^{+\infty} \frac{\lambda_{j,k}}{j!}(s-1)^j ; 
\label{Laurexpk}
\end{equation}
be the Taylor expansion of the regular function $(s-1)^k\zeta^k(s)$ near to $s=1,$ then the rational expansion of $\zeta^k(s),$ which can be considered as a generalization of \eqref{fourexp}, is given in the following theorem.
\begin{theorem}
For any given $k \in \mathbb{N}$ and for all complex number $s=\sigma + it \neq 1$ in the half-plane $\sigma > 1/2,$ we have
$$\zeta^k(s) = \sum_{n \geq -k}(-1)^n \ell_{n,k} \left( \frac{s-1}{s}\right)^n; $$
where
$$\ell_{n,k} := \begin{cases}\begin{array}{cl} \displaystyle (-1)^n \sum_{j=1}^n\binom{n-1}{j-1}\frac{\lambda_{j+k,k}}{(j+k)!} &\mbox{if} \quad n \geq 1, \\  \displaystyle (-1)^k \sum_{j=0}^{k+n}\binom{k-j}{-n} (-1)^j \frac{\lambda_{j,k}}{j!} &\mbox{if} \ -k \leq n \leq 0. \end{array}\end{cases} $$
\label{Th1}
\end{theorem}
Remark that the series in the theorem above is absolutely convergent for all $\sigma > 1/2$ ($s\neq 1$). Moreover, one can obtain the expression of $(\lambda_{j,k})_{j \in \mathbb{N}_0},$ for each $k\in \mathbb{N},$ in terms of Stieltjes constants by applying Cauchy product, \cite[p. 32]{Titch2}, to the absolutely convergent series 
$$(s-1)\zeta(s) =  \sum_{j=0}^{+\infty} \frac{\lambda_j}{j!}(s-1)^j; \qquad |s-1| \leq 1  $$  
where $\lambda_0=1$ and $\lambda_j=(-1)^{j-1}j\gamma_{j-1}$ for $j\in \mathbb{N}.$ Namely, we have $\lambda_{j,1}:=\lambda_j$ for all $j\in \mathbb{N}_0$ and
$$\lambda_{j,k} = \sum_{i=0}^j \binom{j}{i}\lambda_{i,k-1}\lambda_{j-i}, \quad k\geq 2 $$
or equivalently, 
$$\lambda_{0,k}=1 \quad \mbox{and} \quad \lambda_{j,k} =\frac{1}{j} \sum_{i=1}^{j}\binom{j}{i}(ik-j+i)\lambda_{j-i,k}\lambda_i  \quad j\in \mathbb{N}, $$
where $\binom{j}{i} = j!/(i!(j-i)!)$ if $i\in [|0,j|]$ ($j \in \mathbb{N}_0$) and equals $0$ otherwise. Thus, since $|\lambda_j|\leq (\gamma_0)^{j}j!$ for all $j\in \mathbb{N}_0$ then for any given $k \in \mathbb{N}$
$$\frac{|\lambda_{j,k}|}{j!} \leq (\gamma_0)^j\binom{j+k-1}{k-1}; $$
which implies the absolute convergence of the series \eqref{Laurexpk} for all $|s-1|\leq 1.$  

Now, let $\beta_k$ be the order of the sequence $(\ell_{n,k})_n;$ i.e. the least real number such that $\ell_{n,k} = O(n^{\beta_k+\varepsilon})$ for all $\varepsilon > 0$ as $n \to +\infty,$ then it follows by Theorem \ref{Th1} that, for all $\sigma > 1/2$ and $|t|\geq 1$
$$\zeta(s) = O\left( \frac{|s|^{\frac{2}{k}(\beta_k+1)+\varepsilon}}{\left(\sigma - \frac{1}{2}\right)^{\frac{1}{k}}}\right), \qquad \forall\varepsilon > 0. $$
Notice that $-1 \leq \beta_k \leq k(\beta_1+1)-1$ for any given $k \in \mathbb{N};$ hence, the fact that $\beta_1=-1$ implies the Lindel\"{o}f Hypothesis. More generally, if $\limsup|\beta_k/k|=0$ then the Lindel\"{o}f Hypothesis holds true; also, the converse is true thanks to the following corollary.   

\begin{corollary}
The Lindel\"{o}f Hypothesis is true if and only if $(\ell_{n,k})_{n \in \mathbb{N}_0}$ are square-summable sequences for all $k \in \mathbb{N}.$
\label{Coro}
\end{corollary}
\begin{proof}
Let $z=(1-s)/s$ then it is clear that $\sigma > 1/2$ if and only if $z \in \mathbb{D},$ where $\mathbb{D}$ denotes the open unit disk; thus, by Theorem \ref{Th1}, for any $k \in \mathbb{N}$ the function
$$h_k(z) := z^k\zeta^k\left( \frac{1}{1+z} \right) = \sum_{n \geq 0} \ell_{n-k,k}z^n $$
is analytic in $\mathbb{D}.$ Hence, by \cite[Th. 17.12]{Rud}, the sequence $(\ell_{n,k})_{n \in \mathbb{N}_0}$ is square-summable for any given $k\in \mathbb{N}$ if and only if $h_k$ belongs to Hardy space $H^2(\mathbb{D})$ for any given $k \in \mathbb{N};$ namely,
$$\|h_k\|_{H^2}^2:= \sup_{0\leq r<1}\frac{1}{2\pi}\int_{-\pi}^{\pi} \left|h_k\left(re^{i\theta}\right)\right|^2\mathrm{d}\theta = \frac{1}{2\pi} \int_{-\pi}^{\pi}\left|h_k\left(e^{i \theta}\right)\right|^2 \mathrm{d} \theta < +\infty $$
for any given $k \in \mathbb{N};$ or, 
$$\|h_k\|_{H^2}^2 = \frac{1}{2\pi} \int_{\Re s = \frac{1}{2}}\frac{|\zeta(s)|^{2k}}{|s|^2} |\mathrm{d}s|= \sum_{n \geq -k} \ell_{n,k}^2 < +\infty \qquad \forall k \in \mathbb{N} $$
which is equivalent, by combining \cite[Th. 12.5.]{Titch} and \cite[Th. 13.4.]{Titch}, to the truth of the Lindel\"{o}f Hypothesis. 
\end{proof}
\noindent We should not forget to mention that, if $(\ell_{n,k})_{n}$ is square-summable for a given $k \in \mathbb{N}$ then the associated function $h_k,$ defined in the proof of Corollary \ref{Coro}, belongs to the Hardy space $H^2(\mathbb{D})$ and consequently, by \cite[17.11]{Rud} and \cite{Car}, the series in Theorem \ref{Th1} converges almost everywhere in the critical line. However, the following theorem shows that this convergence holds compactly.
\begin{theorem}
If the sequence $(\ell_{n,k})_n$ is square-summable for a given $k \in \mathbb{N},$ then for all $t\in \mathbb{R}$ we have
\begin{equation}
\zeta^k\left( \frac{1}{2}+it \right) = \sum_{n \geq -k}\ell_{n,k}\left( \frac{\frac{1}{2}- it}{\frac{1}{2}+it}\right)^n .
\label{forexpkcl}
\end{equation}
In particular, if the Lindel\"{o}f Hypothesis is true then the expansion \eqref{forexpkcl} holds for every $k \in \mathbb{N}$ and any $t \in \mathbb{R}.$
\label{Th2}
\end{theorem}
\noindent Notice that, the series \eqref{forexpkcl} is conditionally convergent even if the sequence $(\ell_{n,k})_n$ is square-summable. Indeed, the convergence of $ \sum_{n \geq -k}|\ell_{n,k}|$ implies, by \eqref{forexpkcl} and Theorem \ref{Th1}, that $\zeta(s)$ is bounded in the strip $1/2\leq \sigma < 1$ which contradicts the falsity of Lindel\"{o}f's boundedness conjecture \cite[p. 184]{Edw}.  
 \subsection{Proof of theorems}    
We recall, for the sake of completeness, that $\zeta^k(s) = \sum_{n \geq 1}d_k(n)/n^s$ for all $\sigma > 1,$ where $d_k(n)$ denotes the number of expressions of $n \in \mathbb{N}$ as a product of $k$ factors; in particular $d_k(1)=1$ and $d_1(n)=1$ for any positive integer $n.$ Thus, by Abel's summation formula, we have for all $\sigma > \sigma_k$ and $s\neq 1$
\begin{equation}
\zeta^k(s) = s\sum_{j=0}^{k-1}\frac{ a_{j,k}}{(s-1)^{j+1}} + s\int_1^{+\infty} \frac{\Delta_k(x)}{x^{s+1}}  \mathrm{d}x;
\label{intrepzk}
\end{equation}
where $(1-k)/(2k) \leq \sigma_k \leq (k-1)/(k+1)$ is the average order of the error term in the divisor problem $\Delta_k;$ i.e. the least real number such that $\int_1^X\Delta_k^2(x)\mathrm{d}x = O(X^{2\sigma_k+1+\varepsilon})$ for any $\varepsilon >0$ ( see for example \cite[p. 322]{Titch}), 

$$\Delta_k(x) = \left(\sum_{ 1\leq n \leq x}d_k(n)\right) - xP_k\left(\log(x)\right) $$
and $ P_k(X) = \sum_{j=0}^{k-1}(a_{j,k}/j!)X^j$ is a polynomial of degree $k-1.$ Notice that, by using \eqref{Laurexpk} one can obtain the explicit form of the polynomials $(P_k)_{k \in \mathbb{N}}$ in terms of $(\lambda_{j,k});$ namely,
\begin{equation}
a_{j,k} = (-1)^{k-1-j}\sum_{i=0}^{k-1-j}(-1)^i\frac{\lambda_{i,k}}{i!},\qquad (j=0,\cdots ,k-1).
\label{Lavrik}
\end{equation}

We should not forget to mention that, as we shall show in the proof of Theorem \ref{Th1}, the integral in \eqref{intrepzk} is valid for all $\sigma > 1/2;$ however, its absolutely convergence for all $\sigma > 1/2$ and any $k \in \mathbb{N}$ is still open. Notice that, if $\alpha_k$ denotes the order of $\Delta_k$ then $\alpha_k\geq \sigma_k$ and the Lindel\"{o}f Hypothesis is equivalent to $\sigma_k = (k-1)/(2k)$ (or $\alpha_k \leq 1/2$), for any $k\in \mathbb{N}.$ Notice that, the most interesting part of this paper is given in Section 2, in which we shall show, among other things, that the distribution of values of $\Delta_k$ is strongly related to the Fourier coefficients $(\ell_{n,k})_{n\in \mathbb{N}_0}.$

\subsubsection{Proof of Theorem \ref{Th1}}
 Since the series 
 $$\left( \frac{s-1}{s}\right)\zeta(s) = \sum_{n \geq 0}l_{n}\left( \frac{s-1}{s} \right)^n,  $$
 where $l_0=1$ and $l_n = (-1)^{n-1}\ell_{n-1}$ for all $n \geq 1,$ is absolutely convergent for any complex number $s$ in the half-plane $\sigma > 1/2;$ then, by applying Cauchy product, see for example \cite[p. 32]{Titch2}, we obtain, for any $k\in \mathbb{N},$ 
 \begin{align*}
 \left( \frac{s-1}{s}\right)^k\zeta^k(s) &= \left(\sum_{n \geq 0}l_{n}\left( \frac{s-1}{s} \right)^n\right)^k \\ &= \sum_{n \geq 0}l_{n,k}\left( \frac{s-1}{s} \right)^n
 \end{align*}   
where $l_{n,1}=l_n$ and 
\begin{equation} 
l_{n,k} = \sum_{j=0}^n l_{j,k-1}l_{n-j}, \qquad k\geq 2 .
\label{CPl}
\end{equation}
Namely, for any $\sigma > 1/2$ and $s\neq 1,$
$$\zeta^k(s) = \sum_{n \geq -k} l_{n+k,k}\left( \frac{s-1}{s}\right)^n; $$
thus, by putting $l_{n+k,k}=(-1)^n \ell_{n,k},$ we have for all $\sigma > 1/2$ and $s\neq 1$
\begin{equation}
\zeta^k(s) = \sum_{n \geq -k} (-1)^n \ell_{n,k}\left( \frac{s-1}{s}\right)^n.
\label{lnk}
\end{equation}
Since $(l_n)$ is square-summable, then by applying the Cauchy-Schwarz inequality to \eqref{CPl} and by induction we obtain, for any given integer $k \geq 2,$
$$l_{n,k} = O\left( n^{\frac{k-2}{2}+\varepsilon} \right), \qquad \forall \varepsilon > 0. $$
Thus, the radius of convergence is exactly $1,$ since the Riemann zeta function is not bounded, in particularly, on the critical line. 

Now, since for any complex number $s\neq 1$ and any $j \in \mathbb{N}_0$ 
$$
 \left( \frac{1}{s-1} \right)^j = \left( \frac{s}{s-1}-1\right)^j = \sum_{i\geq 0}\binom{j}{i}(-1)^{j-i}\left( \frac{s}{s-1}\right)^i
$$
and for any complex $s\neq 1$
\begin{align*}
 s\sum_{j=0}^{k-1}\frac{a_{j,k}}{(s-1)^{j+1}} &= \sum_{j=0}^{k-1}\frac{a_{j,k}}{(s-1)^{j+1}} + \sum_{j=0}^{k-1}\frac{a_{j,k}}{(s-1)^{j}} \\ &=  \sum_{j=0}^{k}\frac{a_{j-1,k}+a_{j,k}}{(s-1)^{j}}
\end{align*} 
with the convention that $a_{-1,k}=a_{k,k}=0;$ then, for any complex number $s\neq 1,$
\begin{align*}
 s\sum_{j=0}^{k-1}\frac{a_{j,k}}{(s-1)^{j+1}} &=  \sum_{j=0}^{k}\sum_{i=0}^{k}(a_{j,k}+a_{j-1,k})\binom{j}{i}(-1)^{j-i}\left( \frac{s}{s-1}\right)^{i} \\ &=\sum_{i=1}^{k}\left(\sum_{j=i}^{k}\binom{j}{i}\left((-1)^ja_{j,k}-(-1)^{j-1}a_{j-1,k}\right)\right)(-1)^{i}\left( \frac{s}{s-1}\right)^{i}.
\end{align*} 
Hence, by \eqref{Lavrik}, for any $s\neq 1$
\begin{align*}
 s\sum_{j=0}^{k-1}\frac{a_{j,k}}{(s-1)^{j+1}}  &=\sum_{i=1}^{k}\left(\sum_{j=i}^{k}\binom{j}{i}(-1)^j \frac{\lambda_{k-j,k}}{(k-j)!}\right)(-1)^{i}\left( \frac{s}{s-1}\right)^{i}\\ &= \sum_{i=1}^{k}\left(\sum_{j=0}^{k-i}\binom{k-j}{i}(-1)^{k-i-j} \frac{\lambda_{j,k}}{j!}\right)\left( \frac{s}{s-1}\right)^{i}\\ &:= \sum_{n=-k}^{-1} c_{n+k,k}\left(\frac{s-1}{s} \right)^n
\end{align*} 
where, 
$$c_{n,k} = \sum_{j=0}^{n}\binom{k-j}{k-n}(-1)^{n-j}\frac{\lambda_{j,k}}{j!}, \qquad (n=0,\cdots, k-1) .$$
Therefore, the formula \eqref{intrepzk} can be rewritten , for any given $k \in \mathbb{N}$ and for all $\sigma > \sigma_k$ with $s\neq 1,$ as
\begin{equation}
\zeta^k(s) = \sum_{n=-k}^{-1} c_{n+k,k}\left(\frac{s-1}{s} \right)^n + s\int_1^{+\infty}\frac{\Delta_k(x)}{x^{s+1}} \mathrm{d}x .
\label{Hlbr}
\end{equation}
Since the integral in the right-hand side is absolutely convergent for any complex number $s$ in the half-plane $\sigma > \sigma_k$ (which is a domain containing $s=1$) then it represents an analytic function in the half-plane $\sigma > \sigma_k;$ thus, by \eqref{lnk}, we have
\begin{equation}
\ell_{n,k} = (-1)^n c_{n+k} = (-1)^k \sum_{j=0}^{n+k}\binom{k-j}{-n}(-1)^{j}\frac{\lambda_{j,k}}{j!} \qquad (-k \leq n \leq -1) ;
\label{Exprng}
\end{equation} 
and for all $\sigma > \sigma_k$
\begin{equation}
F_k(s):=s\int_1^{+\infty}\frac{\Delta_k(x)}{x^{s+1}} \mathrm{d}x = \sum_{n=0}^{+\infty}(-1)^n \ell_{n,k}\left( \frac{s-1}{s} \right)^n ,
\label{Extdlt}
\end{equation} 
which extends analytically the left-hand side integral of \eqref{Extdlt} to the half-palne $\sigma > 1/2.$ Remark that, this analytic extension is an important unconditional result since the behaviour of $\Delta_k,$ for all positive integers $k,$ is not yet understood completely.  

Finally, it remains to determinate the explicit forms of $\ell_{n,k}$ for $n \in \mathbb{N}_0.$ Since $F_k$ is analytic in $\sigma > \sigma_k,$ then we have near to $s=1,$
$$F_k(s) = \sum_{j=0}^{+\infty}\frac{F_k^{(j)}(1)}{j!}(s-1)^j $$
where $F_k^{(j)}(1)$ denotes the $j$th derivative of $F_k$ at $s=1;$ and by using \eqref{Lavrik} and \eqref{Laurexpk} we have $$F_k^{(0)}(1)=F_k(1)=\frac{\lambda_{k,k}}{k!}-a_{0,k} = (-1)^k\sum_{m=0}^{k}(-1)^m \frac{\lambda_{m,k}}{m!}  $$  
and for all $j \geq 1$
\begin{align*}
 F_k^{(j)}(1) &= \displaystyle \lim_{s \to 1}\frac{\mathrm{d}^j}{\mathrm{d}s^j}\left(\zeta^k(s) - \sum_{m=1}^{k}\frac{a_{m-1,k}+a_{m,k}}{(s-1)^m}\right)   \\ &=\displaystyle \lim_{s \to 1}\frac{\mathrm{d}^j}{\mathrm{d}s^j}\left(\zeta^k(s) - \sum_{m=0}^{k-1}\frac{\lambda_{m,k}}{m!}\frac{1}{(s-1)^{k-m}}\right)= \frac{\lambda_{j+k,k}}{(j+k)!}j!.
\end{align*}
Similarly, by using \eqref{Extdlt}, we obtain
$$\ell_{0,k} = F_k(1) = (-1)^k\sum_{m=0}^{k}(-1)^m \frac{\lambda_{m,k}}{m!} $$
and for $n \geq 1$
\begin{align*}
F_k^{(n)}(1) &= \displaystyle \lim_{s \to 1}\frac{\mathrm{d}^n}{\mathrm{d}s^n} \sum_{j=0}^{+\infty}(-1)^j \ell_{j,k}\left( \frac{s-1}{s} \right)^j \\&=\displaystyle \lim_{s \to 1} \sum_{j=1}^{n}(-1)^j \ell_{j,k} \frac{\mathrm{d}^n}{\mathrm{d}s^n} \left( \frac{s-1}{s} \right)^j \\&= n!(-1)^n\sum_{j=1}^n\binom{n-1}{j-1} \ell_{j,k};
\end{align*}
then, for all $n \geq 1,$
$$\sum_{j=1}^n\binom{n-1}{j-1} \ell_{j,k} = (-1)^n\frac{\lambda_{n+k,k}}{(n+k)!}; $$
which is equivalent by using binomial transform, as in \cite[p.125]{Ela}, to
$$\ell_{n,k} = (-1)^n\sum_{j=1}^n \binom{n-1}{j-1}\frac{\lambda_{j+k,k}}{(j+k)!} \qquad (n\geq 1). $$
Remark that, the case of $n=0$ can be included in the expression \eqref{Exprng} and the proof of Theorem \ref{Th1} is complete.

\subsubsection{Proof of Theorem \ref{Th2}} 
Let $k \in \mathbb{N}$ such that $(\ell_{n,k})_{n}$ is square-summable. Then by \cite[Th. 17.12]{Rud} the holomorphic function $h_k(z)$ defined in the proof of Corollary \ref{Coro} belongs to the Hardy space $H^2(\mathbb{D})$ and we have
$$\ell_{n-k,k} = \frac{1}{2\pi} \int_{-\pi}^{\pi} h_k\left(e^{i\theta}\right)e^{-in\theta} \mathrm{d}\theta \qquad (n\in \mathbb{N}_0).$$
Let $t_0\in \mathbb{R}$ then, by putting $z_0 = (1-s_0)/s_0$ where $s_0 = 1/2 +it_0,$ we have for any $N\in \mathbb{N}$
$$\sum_{n=0}^N\ell_{n-k,k}z_0^n - h_k(z_0) = \frac{1}{2\pi}\int_{-\pi}^{\pi}\frac{h_k\left(e^{i\theta}\right)-h_k(z_0)}{1-z_0e^{-i\theta}}\left(1- (z_0e^{-i\theta})^{N+1}\right)\mathrm{d}\theta . $$
Since $|\arg(z_0)|<\pi$ and $\theta \mapsto h_k(e^{i \theta})$ is differentiable on $(-\pi, \pi)$ then the function 
$$g_k(\theta) := \frac{h_k(e^{i \theta}) - h_k(z_0)}{1-z_0e^{-i \theta}} $$
is square-integrable on $(-\pi , \pi);$ hence,
$$\lim_{N\to +\infty} \frac{1}{2\pi}\int_{-\pi}^{\pi}g_k(\theta)e^{-i(N+1)\theta}\mathrm{d}\theta = 0 $$ 
which implies that
$$\sum_{n=0}^{+\infty}\ell_{n-k,k}z_0^n - h_k(z_0) = \frac{1}{2\pi} \int_{-\pi}^{\pi}g_k(\theta) \mathrm{d}\theta. $$  
Now, by substituting $t=-\tan(\theta/2)/2,$ we obtain
\begin{align*}
\frac{1}{2\pi}\int_{-\pi}^{\pi}g_k(\theta) \mathrm{d}\theta &= \frac{s_0}{2\pi i} \int_{-\infty}^{+\infty}\frac{Z_k\left( \frac{1}{2}+it \right)-Z_k(s_0)}{(\frac{1}{2}+it)(t_0-t)}\mathrm{d}t \\ &=\frac{s_0}{2\pi i}\int_{\Re s = \frac{1}{2}}\frac{Z_k(s)-Z_k(s_0)}{s(s_0-s)}\mathrm{d}s,
\end{align*}
 where, for the reason of simplification,
$$Z_k(s) := h_k\left( \frac{1-s}{s}\right) = \left( \frac{1-s}{s}\right)^k\zeta^k(s). $$
Since the integrand is holomorphic in the half-plane $\sigma \geq 1/2,$ then by Cauchy's integral theorem 
$$\frac{s_0}{2\pi i} \oint_{\mathcal{C}_{R,T}}\frac{Z_k(s)-Z_k(s_0)}{s(s_0-s)}\mathrm{d}s = 0 $$
where $\mathcal{C}_{R,T}$ denotes the counter-clockwise oriented rectangular contour with vertices $1/2+iT,$ $1/2-iT,$ $R-iT$ and $R+iT$ where $R\geq 2$ and $T > 2|t_0|$ are sufficiently large numbers; thus,
$$
\frac{s_0}{2\pi i} \int_{-T}^{T}\frac{Z_k\left( \frac{1}{2}+it \right)-Z_k(s_0)}{(\frac{1}{2}+it)(t_0-t)}\mathrm{d}t  = I(R,T) - J(R,T) + J(R,-T)
$$ 
where
$$I(R,T) = \frac{s_0}{2\pi } \int_{-T}^{T}\frac{Z_k\left( R+it \right)-Z_k(s_0)}{(R+it)\left(s_0 -R-it\right)}\mathrm{d}t  $$
and
$$J(R,T) = \frac{s_0}{2\pi i} \int_{\frac{1}{2}}^{R}\frac{Z_k\left( \sigma +iT \right)-Z_k(s_0)}{(\sigma +iT)\left(s_0-\sigma-iT\right)}\mathrm{d}\sigma  .$$
Thus, by using the Cauchy-Schwarz inequality and the fact that, $|Z_k(R+it)|\leq \zeta^k(2)$ for any $t \in \mathbb{R}$ we have uniformly, for all $T> 2|t_0|,$
\begin{align*}
|I(R,T)| &\leq \frac{|s_0|(\zeta^k(2)+|\zeta^k(s_0)|)}{2\pi} \int_{-T}^{T}\frac{\mathrm{d}t}{|R+it||s_0-R-it|} \\&=O\left( \frac{1}{2R-1} \right).
\end{align*}
Also, since
$$Z_k(\sigma \pm iT) \leq \begin{cases}\begin{array}{ccl}C_k^k \frac{|2 + iT|}{\sqrt{2\sigma - 1}}  &\mbox{if}&  \frac{1}{2}<\sigma < 2 \\ \zeta^k(2)& \mbox{if} &\sigma \geq 2 \end{array}\end{cases}$$
then, we have uniformly
$$|J(R,\pm T)| = O\left( \frac{1}{T} \right).$$  
Therefore, by letting $T\to +\infty$ and $R\to +\infty$ we obtain
$$\frac{1}{2\pi}\int_{-\pi}^{\pi}g_k(\theta) \mathrm{d}\theta = \frac{s_0}{2\pi i} \int_{-\infty}^{+\infty}\frac{Z_k\left( \frac{1}{2}+it \right)-Z_k(s_0)}{(\frac{1}{2}+it)(t_0-t)}\mathrm{d}t=0$$
which completes the proof of Theorem \ref{Th2}.
\section{Extension to the Piltz divisor problem}
Let us start with the following integral representation of the Fourier coefficients $(\ell_{n,k})_{n \in \mathbb{N}_0}.$
\begin{proposition}
For any given $k \in \mathbb{N}$ and for all $n\in \mathbb{N}_0,$ we have
$$(-1)^n \ell_{n,k} = \int_1^{+\infty} \Delta_k(x)\mathcal{L}_n(x)\mathrm{d}w(x), $$
where $(\mathcal{L}_j)_{j \in \mathbb{N}_0}$ is the orthonormal basis in the Hilbert space $\mathcal{H}_0$ defined in Section 1. 
\label{propint}
\end{proposition}
\begin{proof}
Let $k \in \mathbb{N}.$ For $n=0,$ we have
$$\ell_{0,k} = F_k(1) = \int_1^{+\infty}\Delta_k(x) \mathrm{d}w(x) = \int_1^{+\infty} \Delta_k(x) \mathcal{L}_0(x) \mathrm{d}w(x); $$
where $F_k(s)$ is defined in the proof of Theorem \ref{Th1} and for all $n\in \mathbb{N}_0$, $x\geq 1$ 
$$\mathcal{L}_n(x) = \sum_{j=0}^n \binom{n}{j}\frac{(-1)^j}{j!}\log^j(x); $$
see \cite{Ela2} for more details about $(\mathcal{L}_n)_{n\in \mathbb{N}_0}$. For $n \in \mathbb{N},$ we have
$$(-1)^n\ell_{n,k} = \sum_{j=1}^n\binom{n-1}{j-1} \frac{\lambda_{j+k,k}}{(j+k)!} = \sum_{j=1}^n\binom{n-1}{j-1} \frac{F_k^{(j)}(1)}{j!}  $$
and since 
 $$F_k^{(j)}(1) = (-1)^j \int_1^{+\infty}\Delta_k(x)\left(\log^j(x) - j \log^{j-1}(x)\right)\mathrm{d}w(x) \qquad (j \in \mathbb{N}) $$ 
then
$$(-1)^n\ell_{n,k} = \int_1^{+\infty}\Delta_k(x)\sum_{j=1}^n\binom{n-1}{j-1}(-1)^j \left(\frac{\log^j(x)}{j!} -  \frac{\log^{j-1}(x)}{(j-1)!}\right)\mathrm{d}w(x);$$
and the Pascal identity
$$\binom{n-1}{j-1}+ \binom{n}{j-1} = \binom{n}{j}$$
completes the proof.
\end{proof}
Therefore, the error term function $\Delta_k$ belongs to $\mathcal{H}_0,$ for a given $k\in \mathbb{N},$ if and only if $(\ell_{n,k})_{n \in \mathbb{N}_0}$ is square-summable; and we have 
$$\|\Delta_k\|_{\mathcal{H}_0}^2 = \int_1^{+\infty}\left(\frac{\Delta_k(x)}{x}\right)^2\mathrm{d}x = \sum_{n \geq 0}\ell_{n,k}^2 .$$
Moreover, if $\Delta_k \in \mathcal{H}_0,$ for some $k\in \mathbb{N},$ then the equality 
\begin{equation}
 \Delta_k(x) = \sum_{n=0}^{+\infty}(-1)^n\ell_{n,k}\mathcal{L}_n(x) 
\label{FLexp}
\end{equation}
holds in $\mathcal{H}_0$ and almost everywhere on $(1,+\infty),$ by \cite{Car} and \cite[Th. 1]{Muc} or \cite[Th. 9.1.5]{Sz}; thus, by \cite[Th. 12.8]{Titch}, the series in \eqref{FLexp} converges almost everywhere for any $k\in [|1,4|].$ Remark that, if the series in \eqref{FLexp} converges at some $x\geq 1$ to $\Delta_k\in \mathcal{H}_0$ then, by \cite[Th. 8.22.1]{Sz}, there exists $n_0 \in \mathbb{N}$ such that
$$\Delta_k(x) = \frac{\sqrt{x}}{\sqrt{\pi}\log^{\frac{1}{4}}(x)} \sum_{n > n_0}\frac{(-1)^n\ell_{n,k}}{n^{\frac{1}{4}}}\cos\left( 2\sqrt{n\log(x)}- \frac{\pi}{4} \right) + Q_{k,n_0}(\log(x))+O(1) $$
where $Q_{k,n_0}$ is a polynomial of degree $n_0$ and the error term depends only on $k.$ However, even if the series in \eqref{FLexp} converges pointwisely, the convergence could not be uniform since the uniform limit of continuous functions is continuous.

Now, since $\Delta_k \in \mathrm{L}^1(\mathrm{d}w(x)),$ for any $k \in \mathbb{N},$ then, by \cite{Muc2}, its Poisson integral
$$\psi_k(x,\rho) = \int_1^{+\infty}\Delta_k(y)K(x,y,\rho)\mathrm{d}w(y) \qquad x\geq 1, \ \rho \in [0,1)$$
exists for all $\rho \in [0,1)$ and converges almost everywhere to $\Delta_k(x)$ as $\rho \to 1^{-};$ where
$$K(x,y,\rho) = \sum_{n=0}^{+\infty}\mathcal{L}_n(x)\mathcal{L}_n(y)\rho^n = \frac{(xy)^{-\frac{\rho}{1-\rho}}}{1-\rho} I_0\left( \frac{2\sqrt{\rho \log(x)\log(y)}}{1-\rho}\right)  $$   
and $I_0(2v) = \sum_{n \geq 0}v^{2n}/(n!)^2$ denotes the modified Bessel function of the first kind of order $0.$ Thus, we obtain the following result. 

\begin{theorem}
For any given $k\in \mathbb{N}$ and almost all $x\geq 1,$ we have
$$\Delta_k(x) = \lim_{\rho \to 1^-} \sum_{n=0}^{+\infty}(-1)^n\ell_{n,k}\mathcal{L}_n(x)\rho^n.$$
Moreover, the Lindel\"{o}f Hypothesis holds true if and only if the sequences $(\ell_{n,k}^2)_{n \geq 0}$ are Abel summable for every $k\in \mathbb{N}.$
\end{theorem}
\begin{proof}
 It follows, by Proposition \ref{propint} and \cite[Lem. 4]{Muc2}, that for each $\rho\in [0,1)$ and any given $x\geq 1$ and $k\in \mathbb{N},$ the Poisson integral $\psi_k(x,\rho)$ has the expansion  
\begin{equation}
\psi_k(x,\rho) = \sum_{n=0}^{+\infty}(-1)^n\ell_{n,k}\mathcal{L}_n(x)\rho^n .
\label{PoisLag}
\end{equation}
Notice that, by \cite[Th. 8.22.1]{Sz} and the fact that $\ell_{n,k} = O(n^{(k-1)/2+ \varepsilon}),$ $\forall\varepsilon > 0 ,$ the series in \eqref{PoisLag} is absolutely convergent for all $\rho \in [0,1)$ and any $x\geq 1.$ Hence, by \cite[Th. 3]{Muc2} and since $\Delta_k \in \mathrm{L}^1(\mathrm{d}w(x)),$ the following limit holds almost everywhere
$$\Delta_k(x) = \lim_{\rho \to 1^-} \sum_{n=0}^{+\infty}(-1)^n\ell_{n,k}\mathcal{L}_n(x)\rho^n.$$

Moreover, since the sequences $(\ell_{n,k}\rho^n)_{n \geq 0}$ are square-summable for every $k \in \mathbb{N}$ and all $\rho \in [0,1)$ then $\psi_k(\cdot , \rho) \in \mathcal{H}_0$ and for all $\rho \in [0,1)$
$$\|\psi_k(\cdot, \rho)\|_{\mathcal{H}_0}^2 = \sum_{n = 0}^{+\infty}\ell_{n,k}^2 \rho^{2n}. $$
So that, the Lindel\"{o}f Hypothesis is true if and only if $ \|\psi_k(\cdot, \rho)\|_{\mathcal{H}_0}$ converges as $\rho \to 1^-,$ for every $k \in \mathbb{N};$ in this case, the limit must be $\|\Delta_k\|_{\mathcal{H}_0}.$ 
\end{proof}

Finally, we consider the function $\phi_k$ defined, for any given $k \in \mathbb{N},$ on $[0,+\infty)$ by 
$$\phi_k(x) = \int_1^{+\infty} \Delta_k(r) J_0\left(2\sqrt{x\log(r)} \right)\mathrm{d}w(r); $$
where $J_0$ denotes the Bessel function of the first kind and order $0;$ see for example \cite[p. 15]{Sz}. We point out that the integral in the right-hand side is absolutely convergent for all $x\geq 0$ and it follows by the asymptotic formula for $J_0$ given in \cite[eq. 1.71.7]{Sz} that the function $\phi_k$ is continuous and bounded for any $k\in \mathbb{N};$ more precisely, we have
$$\phi_k(x) = O\left( x^{-\frac{1}{4}}\right) $$ 
uniformly as $x \to +\infty .$ Moreover, by using \cite[eq. 5.1.16]{Sz} and Proposition \ref{propint} we obtain, for all $x \geq 0$ and any $k\in \mathbb{N},$ 
$$\phi_k(x) = e^{-x}\sum_{n=0}^{+\infty}\frac{(-1)^n\ell_{n,k}}{n!}x^n $$
which implies that the sequences $(\ell_{n,k})$ converge to $0$ in the Borel sense, for any positive integer $k.$ Notice that, the radius of convergence of the series above is $+\infty,$ since $(\ell_{n,k})_n$ has a polynomial order. \\
In fact, the function $\phi_k$ represents a modified form of Hankel transform of the function $\Delta_k(r)/r,$ for each $k\in \mathbb{N};$ which is invertible and its inverse is given, for almost all $r\geq 1$ and any given $k\in \mathbb{N},$ by
$$\frac{\Delta_k(r)}{r} = \int_0^{+\infty}\phi_k(x)J_0\left(2\sqrt{x\log(r)}\right)\mathrm{d}x .$$
Thus, by Parseval's identity we have, for any $k \in \mathbb{N},$
$$ \|\phi_k\|_{2}^2:= \int_0^{+\infty} \phi_k^2(x) \mathrm{d}x = \int_1^{+\infty}\frac{\Delta_k^2(r)}{r^3}\mathrm{d}r < +\infty, $$
which implies that $\phi_k \in \mathrm{L}^2(\mathbb{R}_+).$ Hence, if we denote by $\phi_{N,k}(x)$ the partial sum of $\phi_k(x);$ i.e.  $\phi_{N,k}(x)= e^{-x}\sum_{n=0}^N(-1)^n\ell_{n,k}/n! x^n,$ then we obtain the following result.
\begin{theorem}
The Lindel\"{o}f Hypothesis is true if and only if $(\phi_{N,k})_{N\in \mathbb{N}_0}$ converges in $\mathrm{L}^2(\mathbb{R}_+)$ to $\phi_k,$ for any $k\in \mathbb{N}.$
\end{theorem}  
\begin{proof}
Let $k \in \mathbb{N}.$ It follows by \cite[Th. 5.4]{Sz} that for all $r\geq 1$ and any $N\in \mathbb{N}$
$$\frac{1}{r}\sum_{n=0}^N(-1)^n\ell_{n,k}\mathcal{L}_n(r)  = \int_0^{+\infty}\phi_{N,k}(x) J_0\left(2\sqrt{x\log(r)}\right)\mathrm{d}x ;$$
thus, for almost all $r\geq 1$ we have
$$\frac{\Delta_k(r)}{r} - \frac{1}{r}\sum_{n=0}^N(-1)^n\ell_{n,k}\mathcal{L}_n(r)  = \int_0^{+\infty}\left(\phi_k(x)-\phi_{N,k}(x)\right) J_0\left(2\sqrt{x\log(r)}\right)\mathrm{d}x .$$
Then, by Parseval theorem we obtain
$$\|\phi_k - \phi_{N,k}\|_2^2 = \int_1^{+\infty}\left|\Delta_k(r) -  \sum_{n=0}^N(-1)^n\ell_{n,k}\mathcal{L}_n(r)\right|^2\frac{\mathrm{d}r}{r^3}. $$
Therefore, if the Lindel\"{o}f Hypothesis is true then
$$\|\phi_k - \phi_{N,k}\|_2 \leq \left\|\Delta_k - \sum_{n=0}^N(-1)^n\ell_{n,k}\mathcal{L}_n \right\|_{\mathcal{H}_0} ;$$
which implies the convergence of $(\phi_{N,k})_{N\in \mathbb{N}_0}$ in $\mathrm{L}^2(\mathbb{R}_+)$ to $\phi_k,$ for any $k\in \mathbb{N}.$

Reciprocally, we assume that $(\phi_{N,k})_{N\in \mathbb{N}_0}$ in $\mathrm{L}^2(\mathbb{R}_+)$ to $\phi_k,$ for any $k\in \mathbb{N}$ and let $\sigma \in (0,1/2);$ then by Cauchy-Schwarz inequality, we have
\begin{align*}\int_0^{+\infty}|\phi_k(x) - \phi_{N,k}(x)|x^{\sigma - 1}\mathrm{d}x &= \left\{\int_0^1 + \int_1^{+\infty}\right\}|\phi_k(x) - \phi_{N,k}(x)|x^{\sigma - 1}\mathrm{d}x \\&\leq \sum_{n=N}^{+\infty}\frac{|\ell_{n,k}|}{n!} + \frac{\left\| \phi_k - \phi_{N,k} \right\|_2}{\sqrt{1-2\sigma}}.
\end{align*}   
Since, for all $\Re(s)=\sigma \in (0,1/2)$
\begin{align*}
\left|\tilde{\phi}_k(s) - \sum_{n=0}^N\frac{(-1)^n\ell_{n,k}}{n!}\Gamma(s+n)\right| &= \left| \int_0^{+\infty}(\phi_k(x) - \phi_{N,k}(x))x^{s - 1}\mathrm{d}x\right| \\ &\leq \sum_{n=N}^{+\infty}\frac{|\ell_{n,k}|}{n!} + \frac{\left\| \phi_k - \phi_{N,k} \right\|_2}{\sqrt{1-2\sigma}},
\end{align*}
where $\tilde{\phi}_k$ denotes the Mellin transform of $\phi_k;$ then, for all $\sigma \in (0,1/2)$ and any given $k\in \mathbb{N},$
$$\sum_{n=0}^{+\infty} \frac{(-1)^n\ell_{n,k}}{n!}\Gamma(s+n) = \tilde{\phi}_k(s) $$
which implies, for any given $k \in \mathbb{N},$ that
$$\ell_{n,k} = o\left( \frac{n!}{\Gamma(\sigma+n)}\right) = o(n)  $$
uniformly as $n \to +\infty;$ namely, $-1\leq \beta_k < 1,$ and hence the Lindel\"{o}f Hypothesis holds. 

Remark that the Mellin transform $\tilde{\phi}_k(s)$ is well-defined and analytic in the strip $0< \sigma < 1/2$ since $\phi_k$ is continuous and $\phi_k(x)=O(x^{-1/4})$ as $x\to +\infty.$ Moreover, by using \cite[eq. 7.4.1]{Titch3}, we obtain, for all $\sigma \in (0,1/4)$ and any given $k\in \mathbb{N},$ 
$$\tilde{\phi}_k(s) = \frac{\Gamma(s)}{\Gamma(1-s)}\int_1^{+\infty}\frac{\Delta_k(r)}{\log^s(r)} \mathrm{d}w(r) ;$$
which is valid for all $\sigma \in (0,1)$ and extends $\tilde{\phi}_k(s)$ analytically to the half-plane $\sigma < 1$ except at simple poles $s=-m$ ($m \in \mathbb{N}_0).$ Actually, one can show that the analytic extension of $\tilde{\phi}_k$ holds in the whole complex plane except at $s=-m.$   
\end{proof}


\begin{thebibliography}{00}

\bibitem{Ber}
B.C. Berndt, On the Hurwitz zeta-function, Rocky Mountain J. Math. 2 (1972) 151–-157.

\bibitem{Bour}
J. Bourgain, Decoupling, exponential sums and the Riemann zeta function, J. Amer. Math. Soc. 30 (2017) 205--224.

\bibitem{Car}
L. Carleson, On Convergence and growth of partial sums of Fourier series, Acta Mathematica, 116 (1) (1966) 135--157.

\bibitem{Edw}
H.M. Edwards, Riemann's zeta function, Academic Press, New York (1974).

\bibitem{Ela2}
L. Elaissaoui, On expansions involving the Riemann zeta function and its derivatives, J. Math. Anal. Appl., 517 (2023) 126569.

\bibitem{Ela}
L. Elaissaoui, Z. E. Guennoun, Fourier expansion of the Riemann zeta function and applications, J. of Number Theory, Elsevier, 211 (2020) 113--138.

\bibitem{HarLit}
G.H. Hardy, J.E. and Littlewood, On Lindel\"{o}f's hypothesis concerning the Riemann zeta-function, Proc. Royal Soc. 103 (1923), 403--412.

\bibitem{Muc}
B. Muckenhoupt, Equiconvergence and almost everywhere convergence of Hermite and Laguerre series, SIAM J. Math. Anal. 1 (1970), 295--321.

\bibitem{Muc2}
B. Muckenhoupt, Poisson integrals for Hermite and Laguerre expansions, Tans. Amer. Math. Soc. 139 (1969), 231--242.

  \bibitem{Rud}
  W. Rudin, Real and Complex Analysis, third ed., McGraw-Hill Book Company, New York (1987)  

\bibitem{Sz}
G. Szeg\"o, Orthogonal Polynomials, 4th ed. Amer. Math. Soc. Providence Rhode Island, U.S.A. (1975).


\bibitem{Titch}
E.C. Titchmarsh, The Theory of the Riemann Zeta-Function, 2nd ed. Oxford University Press, New York (1986).
 
\bibitem{Titch2}
E.C. Titchmarsh, The Theory of Functions,  2nd ed. Oxford University Press, Oxford (1939).

\bibitem{Titch3}
E.C. Titchmarsh, Introduction to the Theory of Fourier Integrals, 2nd ed. Oxford University Press, New York (1948).


\end{thebibliography}
\end{document}